\documentclass[article]{elsarticle}
\usepackage[all]{xy}
\usepackage{color}
\usepackage{url}
\usepackage{amscd, amssymb, amsfonts, amsthm, amsmath}
\usepackage{mathdots}
\usepackage{hyperref}

\newcommand{\soc}{\mbox{\rm{Soc}}}
\newcommand{\topp}{\mbox{\rm{Top}}}
\newcommand{\length}{\mbox{\rm{l}}}
\newcommand{\start}{\mbox{\rm{s}}}
\newcommand{\target}{\mbox{\rm{t}}}
\newcommand{\fin}{\mbox{\rm{findim}}}
\newcommand{\prj}{\mbox{\rm{P}}}
\newcommand{\fidim}{\mbox{\rm{$\phi$dim}}}
\newcommand{\psidim}{\mbox{\rm{$\psi$dim}}}
\newcommand{\fdim}{\mbox{\rm{findim}}}
\def\ker{\mbox{\rm{Ker}}}

\def\mod{\mbox{\rm{mod}}}
\def\ind{\mbox{\rm{ind}}}
\def\add{\mbox{\rm{add}}}

\def\pd{\mbox{\rm{pd}}}

\def\gd{\mbox{\rm{gldim}}}

\def\rk{\hbox{\rm{rk}}}

\begin{document}
\newcommand{\mono}[1]{%
\gdef\puA{#1}}
\newcommand{\puA}{}
\newcommand{\faculty}[1]{%
\gdef\puC{#1}}
\newcommand{\puC}{}
\newcommand{\facultad}[1]{%
\gdef\puD{#1}}
\newcommand{\puD}{}
\newcommand{\N}{\mathbb{N}}
\newcommand{\Z}{\mathbb{Z}}
\newcommand{\gl}{\mbox{\rm{gldim}}}
\newtheorem{teo}{Theorem}[section]
\newtheorem{prop}[teo]{Proposition}
\newtheorem{lema}[teo] {Lemma}
\newtheorem{ej}[teo]{Example}
\newtheorem{obs}[teo]{Remark}
\newtheorem{defi}[teo]{Definition}
\newtheorem{coro}[teo]{Corollary}
\newtheorem{nota}[teo]{Notation}

\newenvironment{note}{\noindent Notation: \rm}

%\newenvironment{proof}{\noindent {\it\bf Proof: \newline}\rm}

%\predate{\begin{flushright}\large\scshape}
%\postdate{\par\end{flushright}}

\title{Gaps for the Igusa-Todorov function}

\author{Marcos Barrios}
\address{Universidad de La Rep\'ublica, Facultad de Ingenier\'ia -  Av. Julio Herrera y Reissig 565, Montevideo, Uruguay}
\ead{marcosb@fing.edu.uy}

\author{Gustavo Mata}
\address{Universidad de la Rep\'ublica, Facultad de Ingenier\'ia -  Av. Julio Herrera y Reissig 565, Montevideo, Uruguay} 
\ead{gmata@fing.edu.uy}

\author{Gustavo Rama}
\address{Universidad de la Rep\'ublica, Facultad de Ingenier\'ia -  Av. Julio Herrera y Reissig 565, Montevideo, Uruguay} 
\ead{grama@fing.edu.uy}

\begin{abstract}
For a finite dimensional algebra $A$ with $0 < \fidim (A) = m < \infty$ we prove that there always exist modules $M$ and $N$ such that $\phi(M) = m-1$ and $\phi (N) = 1$.
On the other hand, we see an example of an algebra that not every value between $1$ and its $\phi$-dimension is reached by the $\phi$ function. We call that values gaps and we prove that the algebras with gaps verifies the finitistic  conjecture.

\end{abstract} 


\maketitle

\section{Introduction}

In an attempt to prove the Finitistic Conjecture, Igusa and Todorov defined in \cite{IT} two functions from the objects of $\mod A$ (the category of right finitely generated modules over an Artin algebra $A$) to the natural numbers, which generalizes the notion of projective dimension. Nowadays they are known as the Igusa-Todorov functions, $\phi$ and $\psi$.
Several recent works dedicated to the study of this functions and their associated dimensions shows the growing relevance of the subject. Also, a better knowledge about the Igusa-Todorov functions could help to a more profound understanding of certain homological conjectures, such as the Finitistic Conjecture.

This article is organized as follows: the the preliminary section are devoted to fixing the notation and recalling the basic facts needed in this article.
In section $3$ we prove that certain functors ralate the $\phi$ and $\psi$ functions between two algebras.
In section $4$ we prove the following theorem regarding the existence of some values for the $\phi$ function for any algebra.\\

\underline{\bf{Theorem:}} Let $A$ be a finite dimensional algebra. If $\fidim (A) > 0$ then there exist $M\in \mod A$ such that $\phi (M) = 1$. In the case that $\fidim(A)$ is finite, then there exist $N\in  \mod A$ such that $\phi (N) = \fidim(A) - 1$.\\

We also see in Example \ref{ejemplo_gap} that the theorem is not true for others values between $1$ and $\fidim(A)$. If, for an algebra $A$, there is a natural number $\lambda$ not reached by the $\phi$ function, we say that $A$ has a gap. For the case of algebras with gaps we prove the following theorem.\\

\underline{\bf{Theorem:}} If $A$ is a finite dimensional algebra with gaps, then:
$$\fdim(A)<\lambda< \fidim(A),$$
for every gap $\lambda$. Hence the finitistic conjecture holds for $A$.\\

On the other hand, we give conditions in quivers such that its associated radical square zero algebra the $\phi$ function reaches every value between $0$ and $\fidim(A)$. Finally, in section $5$, we introduce the partial $\phi$-dimension and use them to characterize the selfinjective algebras.

\section{Preliminaries}

Let $A$ be a finite dimensional basic algebra defined over a field $\Bbbk$. The category of finite dimensional right $A$-modules will be denoted by $\mod A$, the indecomposable modules of $A$ by $\ind A$, and the set of isoclasses of simple $A$-modules by $\mathcal{S} (A)$. If $S \in \mathcal{S} (A)$, $\prj(S)$ denotes the indecomposable projective associated to $S$. For a $A$-module $M$ we denote by $\soc(M)$ its socle and by $\topp(M)$ its top.\\ 
Given an $A$-module $M$ we denote its projective dimension by $\pd(M)$ and the $n^{th}$-syzygy by $\Omega^n(M)$. We recall that the global dimension of $A$, which we denote by $\gd(A)$, is the supremum  of the set of projective dimensions of $A$-modules. The global dimension can be a natural number or infinity. The finitistic dimension of $A$, denoted by $\fin(A)$, is the supremum of the set of projective dimensions of $A$-modules with finite projective dimension.\\
If $Q$ is a finite connected quiver, $\Bbbk Q$ denotes its path algebra. Given $\rho$ a path in $\Bbbk Q$, $\length(\rho)$, $\start(\rho)$ and $\target(\rho )$ denote the length, start and target of $\rho$ respectively. $J_Q$ denotes the ideal of $\Bbbk Q$ generated by the arrows of $Q$. In case that there is no possible misinterpretation we denote that ideal by $J$. 

The following theorem will be useful.

\begin{teo}[\cite{N}]\label{Nakayama}

Let $A$ be a finite dimensional basic algebra. The following statements are equivalent:
\begin{enumerate}

\item  $A$ is self-injective

\item The rule $S \rightarrow \soc(\prj(S))$ defines a permutation $\nu :\mathcal{S} (A) \rightarrow \mathcal{S}(A).$

\end{enumerate}  

\end{teo}

\subsection{Igusa-Todorov $\phi$ function} 
We recall the definition of the Igusa-Todorov $\phi$ function and some basic properties. We also define the $\phi$-dimension of an algebra.

\begin{defi}
Let $K_0(A)$ be the abelian group generated by all symbols $[M]$, where M is a f.g.
$A$-module, modulo the relations
\begin{enumerate}
  \item $[M]-[M']-[M'']$ if  $M \cong M' \oplus M''$,
  \item $[P]$ for each projective.
\end{enumerate}
\end{defi}

\noindent Let $\bar{\Omega}: K_0 (A) \rightarrow K_0 (A)$ be the group endomorphism induced by $\Omega$, and let $K_i (A) = \bar{\Omega}(K_{i-1}(A))= \ldots = \bar{\Omega}^{i}(K_{0} (A))$. Now, if $M$ is a finitely generated $A$-module then $\langle \add M\rangle$ denotes the subgroup of $K_0 (A)$ generated by the classes of indecomposable summands of $M$.
Given a subgroup $G$ of $K_0(A)$, we define $\eta_{\bar{\Omega}}(G)$ as the first time where $\rk(\bar{\Omega}^n(G))$ stabilizes.

\begin{defi}\label{monomorfismo}
\cite{IT}The \textbf{(right) Igusa-Todorov function} $\phi$ of $M\in \mod A$  is defined as 
\[\phi_{A}(M) = \min\left\{l: \bar{\Omega}_{A} {\vert}_{{\bar{\Omega}_{A}}^{l+s}\langle \add M\rangle}\text{ is a monomorphism for all }s\in \mathbb{N}\right\}.\]
%l:\bar{\Omega}_{A}{|}_{{\bar{\Omega}_{A}}^{l+s}\langle add M\rangle\right\}$ is a monomorphism for all $s \in \mathbb{N}\right\}$.
\end{defi}

In case that there is no possible misinterpretation we denote by $\phi$ the Igusa-Todorov function $\phi_A$.

\begin{prop}[\cite{IT}, \cite{HLM}] \label{it1} \label{Huard1}
Given $M,N\in$ mod$A$. 

\begin{enumerate}
  \item $\phi(M) = \pd (M)$ if $\pd (M) < \infty$.
  \item $\phi(M) = 0$ if $M \in \ind A$ and $\pd(M) = \infty$.
  \item $\phi(M) \leq \phi(M \oplus N)$.
  \item $\phi\left(M^{k}\right) = \phi(M)$ for $k \in \mathbb{N}$.
  \item If $M \in \mod A$, then $\phi(M) \leq \phi(\Omega(M))+1$.
\end{enumerate}
\begin{proof}
For the statements 1-4 see \cite{IT}, and for 5 see \cite{HLM}.
\end{proof}
\end{prop}

\begin{defi}
\cite{IT} The \textbf{(right) Igusa--Todorov function
$\psi_{A}$} of $M\in \mod A$  is defined as 
\begin{equation*}
\psi_{A}(M) = \phi_{A}(M) + \sup\{ \pd(N) : \Omega^{\phi(M)}(M) = N
\oplus N'\mbox{ and } \pd(N)< \infty\}.
\end{equation*}
\end{defi}

In case that there is no possible misinterpretation we denote by $\psi$ the Igusa-Todorov function $\psi_A$.

\begin{prop}
[\cite{IT}, \cite{HLM}]
\label{Huard2}
Given $M,N\in$ mod$A$.
\begin{enumerate}
\item[1.] $\psi(M) = \pd(M)$ if $\pd (M) < \infty$.
\item[2.] $\psi(M) = 0$ if $M \in \ind A$ and $\pd(M) = \infty$.
\item[3.] $\psi(M) \leq\psi(M \oplus N)$.
\item[4.] $\psi(M^{k}) = \psi(M)$ for $k \in\mathbb{N}$.
\item[5.] If $N$ is a direct summand of $\Omega^{n}(M)$ where $n\leq
\phi(M)$ and
$\pd(N) < \infty$, then $\pd(N) + n \leq\psi(M)$.
\item[6.] If $M \in\mathrm{mod}A$, then $\psi(M) \leq\psi(\Omega(M))+1$.
\end{enumerate}
\end{prop}
\begin{proof}
For the statements 1--5 see \cite{IT}, and for 6 see
\cite{HLM}.
\end{proof}

\begin{defi}
Let $A$ be an Artin algebra and $\mathcal{C}$ a full subcategory of
$\mathrm{mod}A$. We define:
\begin{itemize}
\item
$\fidim(\mathcal{C}) = \sup\{\phi(M) $ such that $ M \in Ob
\mathcal{C} \}$,
\item
$\psidim(\mathcal{C}) = \sup\{\psi(M) $ such that $ M \in Ob
\mathcal{C} \}$.
\end{itemize}

In particular, we denote by

\begin{itemize}
\item
$\fidim(A) = \fidim(\mod A)$,
\item
$\psidim(A) = \psidim(\mod A)$.
\end{itemize}

\end{defi}
\begin{teo}[\cite{HL}] \label{Phi = 0}

If $A$ is an Artin algebra then $A$ is self-injective if and only if $\fidim (A) = 0$.

\end{teo}

\subsection{Radical square zero algebras}
Given a quiver Q we denote by  $J$  the ideal generated by the arrows in $\Bbbk Q$.
By a radical square zero algebra we mean an algebra which is isomorphic to an algebra of the type $A=\frac{\Bbbk Q}{J^2}$.
If $\mathcal{S}({A}) = \left\{S_1, \ldots, S_n\right\}$ denotes a complete set of simple $A$-modules up to isomorphism, then we divide the set $\mathcal{S}(A)$ in the three distinct sets

\begin{itemize}
  \item $\mathcal{S}_I$ the set of injective modules in  $\mathcal{S}$,
  \item $\mathcal{S}_P$ the set of projective modules in $\mathcal{S}$,
  \item $\mathcal{S}_D = \mathcal{S}\setminus (\mathcal{S}_I \cup \mathcal{S}_P)$.
\end{itemize}

\begin{obs} \label{fgd}
For radical square zero algebras it holds that  
$\displaystyle \Omega (S_i) = \bigoplus_{\alpha:i\rightarrow j} S_j$, i.e.: $\Omega (S_i)$ is a direct sum  of simple modules, and the number of summands isomorphic to $S_j$ coincides with the number of arrows starting in the vertex $i$ and ending at the vertex $j$. 
Given a radical square zero algebra $A$ with $n$ vertices and finite global dimension, is easy to compute its global dimension using the last fact. Explictly we have
\[ \gd(A) = \sup\left\{\length(\rho): \start(\rho)\text{ is a source, and }\target(\rho)\text{ is a sink}\right\},\]
conluding that $\gd(A)$ must be less or equal to $n-1$.
%Using the last fact it is easy to compute that for a radical square zero algebra $A$ with
%$n$ vertices and finite global dimension, its global dimension must be less or equal to $n-1$.
\end{obs}

\begin{obs}

For a radical square zero algebra $A = \frac{\Bbbk Q}{J^2}$, $\Omega(M)$ is a semisimple $A$-module for every $A$-module $M$, and $K_1 (A)$ has basis $\left\{[S]_{S \in \mathcal{S}_D}\right\}$.
In particular if $Q$ has no sinks nor sources, then $K_1 (A)$ has basis $\{[S]_{S \in \mathcal{S}}\}$.

\end{obs}

\begin{defi}

Let $ A = \frac{\Bbbk Q}{J^2}$ be a radical square zero algebra, where $Q$ is finite with $n$ vertices and without sources nor sinks. We define $T:\mathbb{Q}^n \rightarrow \mathbb{Q}^n$ as the linear transformation given by $T(e_i) =  \sum_{j = 1}^n |\{\alpha:i\rightarrow j\}|e_j$.  

\end{defi}

\begin{obs}
Given a radical square zero algebra $A$ with $Q$ as the previous definition, the matrix of $T$ in the canonical basis and the matrix of $\bar{\Omega}|_{K_1(A)}$ in the canonical basis $\{[S]_{S \in \mathcal{S}}\}$ agree.
\end{obs}

The following result give a way to compute the $\phi$ function on radical square zero algebras.

\begin{prop}\label{tata}\cite{LMM}
If A is a radical square zero algebra not self-injective, then:
$$ \fidim(A) = \phi(\oplus_{S\in\mathcal{S}_{D}}S)+1.$$

\end{prop}

Also we know that, for radical square zero algebras, there is bound for the $\phi$-dimension. 

\begin{prop}\label{toto}\cite{LMM}
If $A$ is a radical square zero algebra then  $\fidim(A)\leq n$ where $|Q_0| = n$.
\end{prop}

\section{Exact additive functors that preserve projective resolutions}

In this section we relate the Igusa-Todorov functions between two categories of modules when there exist  functors with nice properties. 
From now on we will only consider exact and additive functors. We recall that an additive functor $F: \mod A \rightarrow \mod B$ preserves projective modules if and only if $F(A)$ is a direct summand of $B^n$. These functors verifies the next remark.

\begin{obs}\label{diagrama1}

Let $F : \mod A \rightarrow \mod B$ be an exact and additive functor such that $F(A)$ is direct summand of $B^n$, then $F$ induces a group morphism $\bar{F} : K_0(A) \rightarrow K_0(B)$ such that the following diagram is commutative:

$$\xymatrix{K_0(A) \ar[r]^{\bar{F}} \ar[d]_{\bar{\Omega}_A} & K_0(B) \ar[d]^{\bar{\Omega}_B} \\ \ar[r]^{\bar{F}} K_0(A) & K_0(B)}$$

\end{obs}

\begin{teo}\label{teorema iny}
Let $A$ and $B$ be Artin algebras, and $F : \mod A \rightarrow \mod B$ an exact additive functor such that $F(A)$ is direct summand of $B^n$. Assume that $\bar{F}$ is a monomorphism, then:    

\begin{enumerate}

\item  $\phi_A (M) \leq \phi_B(F(M))$ for all $M \in \mod A$,
\item  $\pd_A M = \pd_B F(M)$ for all $M \in \mod A$ such that $\pd_A M < \infty$,
\item  $\psi_A (M) \leq \psi_B(F(M))$ for all $M \in \mod A$.

\end{enumerate}

\begin{proof}\ 
\begin{enumerate}
\item 
We have that $\bar{F} \bar{\Omega}^i (\langle \add M \rangle) = \bar{\Omega}^i \bar{F} (\langle \add M \rangle)$ because of the commutativity from Remark \ref{diagrama1}. On the other hand $\rk \bar{F} \bar{\Omega}^i \langle \add M \rangle = \rk \bar{\Omega}^i \langle \add M \rangle $ because $\bar{F}$ is a monomorphism, then $\eta_{\bar{\Omega}} (\bar{F} \langle \add M \rangle) = \phi_A(M)$. Now, using that $\bar{F}\langle \add M\rangle \subseteq \langle \add FM\rangle$, we get that $\phi_A(M) \leq \phi_B(F(M))$.
\item If $\pd_A M = n < \infty $, then there exist a exact sequence:
$$\xymatrix{ 0 \ar[r] & P_n \ar[r] & \ldots \ar[r] & P_1 \ar[r] & P_0 \ar[r] & M \ar[r] & 0}.$$
Applying the functor $F$, we obtain the following exact sequence:
$$\xymatrix{ 0 \ar[r] & F(P_n) \ar[r] & \ldots \ar[r] & F(P_1) \ar[r] & F(P_0) \ar[r] & F(M) \ar[r] & 0},$$
where each $F(P_i)$ is projective for $i = 0,\dots,n$, then $\pd_B F(M) \leq n$. Using \cite{IT}, the previous statement implies that $\pd_A M \leq \pd_B F(M)$. 
\item We have $$\psi_A(M) = \phi_A(M) + \sup \{ \pd N : \pd N < \infty \ and \ N\oplus N' \cong \Omega_A^{\phi_A(M)} (M)\}.$$
\begin{itemize}
\item If $\psi_A(M) \leq \phi_B(F(M))$ the statement clearly follows.
\item If $\phi_A(M) \leq \phi_B(F(M)) < \psi_A(M)$, there exists $N \in \mod A$ such that $N \oplus N' \cong \Omega_A^{\phi_A(M)} (M) $ with $\pd_A (N) = \psi_A(M)- \phi_A(M)> \phi_B(F(M))-\phi_A(M)$. Then $F(N) \oplus F(N') \cong \Omega_B^{\phi_A(M)} (F(M))$ because of the commutativity from Remark \ref{diagrama1}, in particular there exist an isomorphism $\Omega_B^{\phi_B(F(M))-\phi_A(M)}(F(N) \oplus F(N')) \cong \Omega_B^{\phi_B(F(M))} (F(M))$. On the other hand we get that $\pd ( \Omega_A^{\phi_B(F(M))-\phi_A(M)} (N) ) = \pd_A N + \phi_A(M)-\phi_B(F(M))$, because $\pd_A N = \psi_A(M)- \phi_A(M)> \phi_B(F(M))-\phi_A(M)$.\\
Finally, we obtain $\psi_B(F(M)) \geq \phi_B(F(M))+\pd_A(N) +  \phi_A(M) - \phi_B(F(M)) = \psi_A(M)$.
\end{itemize}
\end{enumerate}
\end{proof}  

\end{teo}

\begin{coro}\label{iny+ind}

Let $A$ and $B$ be Artin algebras, and $F : \mod A \rightarrow \mod B$ be an exact additive functor such that $F$ preserves indecomposable modules and $F(A)$ is direct summand of $B^n$. Assume that $\bar{F}$ is a monomorphism, then:    

\begin{enumerate}

\item  $\phi_A (M) = \phi_B(F(M))$ for all $M \in \mod A$,
\item  $\pd_A M = \pd_B F(M)$ for all $M \in \mod A$ such that $\pd_A M < \infty$,
\item  $\psi_A (M) = \psi_B(F(M))$ for all $M \in \mod A$.

\end{enumerate}

The following example shows that the inequalities in Theorem \ref{teorema iny} could be strict.

\begin{ej}
Let $Q$ and $\bar{Q}$ be the quivers:

$$Q = \xymatrix{1 \ar[r] & 2  \ar@(dr,ur)}\ \ \ \ ,$$

$$\tilde{Q} = \xymatrix{  1 \ar[r] & 2  \ar@(ul,ur) & 3 \ar[l]}.$$

Consider $F_1, F_2:  \mod \frac{\Bbbk Q}{J^2} \rightarrow \mod \frac{\Bbbk \tilde{Q}}{J^2}$ the functors such that:

\begin{itemize}

\item for each representation $\xymatrix{  V_1 \ar[r]^f & V_2  \ar@(dr,ur)_g } \in  \mod \frac{\Bbbk Q}{J^2}$, then $$F_1 \left(\xymatrix{  V_1 \ar[r]^f & V_2  \ar@(dr,ur)_g }\right) =   \xymatrix{  V_1 \ar[r]^f & V_2  \ar@(ul,ur)^g & 0 \ar[l]}, $$
\vspace{.5cm}
and for every morphism $(\alpha, \beta)$, $F_1((\alpha, \beta)) = (\alpha, \beta, 0)$.

\item For each representation $\xymatrix{  V_1 \ar[r]^f & V_2  \ar@(dr,ur)_g } \in  \mod \frac{\Bbbk Q}{J^2}$, then $$F_2 \left(\xymatrix{  V_1 \ar[r]^f & V_2  \ar@(dr,ur)_g }\right) =   \xymatrix{  0 \ar[r] & V_2  \ar@(ul,ur)^g & V_1 \ar[l]_f},$$

\vspace{.5cm}

and for every morphism $(\alpha, \beta)$, $F_2((\alpha, \beta)) = (0, \beta, \alpha)$.

\end{itemize}

For $i = 1,2$, $F_i$ is an exact additive functor such that $F_i\left(\frac{\Bbbk Q}{J^2}\right)$ is a direct summand of $\frac{\Bbbk \tilde{Q}}{J^2}$. Now, consider the functor $F = F_1 \oplus F_2$. $F$ is an exact additive functor such that $F\left(\frac{\Bbbk Q}{J^2}\right)$ is a direct summand of $\left(\frac{\Bbbk \tilde{Q}}{J^2}\right)^2$ and $\bar{F}$ is a monomorphism.
Finally, for $S_1 \in \mod \frac{\Bbbk Q}{J^2}$, $\phi(S_1) = 0$ and $\phi(F(S_1)) = 1$.

\end{ej}

\end{coro}

We obtain the following corollary applying the Theorem \ref{teorema iny}. This result is a generalization of Theorem 3.4 of \cite{M}.

\begin{coro}

Let $A$ and $B$ be Artin algebras, and $F : \mod A \rightarrow \mod B$ an exact additive functor such that $F(A)$ is direct summand of $B^n$ and $\bar{F}$ is a monomorphism. Then $\fidim (A) \leq \fidim (B)$,  $\psidim (A) \leq \psidim (B)$ and $\fdim (A) \leq \fdim (B)$.  

\end{coro}

\begin{defi}

We denote by:

\begin{itemize}

\item $\bar{K}_1(A)$ the full additive subcategory of $\mod A$, where the objects are $M \in \mod A$ such that $[M] \in K_1(A)$.

\item $\bar{K}_i(A)$ for $i\in \{2,3,\ldots \}$ the full additive subcategory of $\mod A$, where the objects are $M \in \mod A$ such that there exist $[N] \in K_i(A)$ with $M\oplus M' \cong N$.

\end{itemize}
\end{defi}

\begin{obs}

Let $G : \bar{K}_i(A) \rightarrow \mod B$ be an exact and additive functor such that $G(P)$ is direct summand of $B^n$ for every $P$ projective $A$-module in $\bar{K}_1(A)$, then $G$ induces a group morphism $\bar{G} : K_i(A) \rightarrow K_0(B)$ such that the following diagram is commutative:

$$\xymatrix{K_i(A) \ar[r]^{\bar{G}} \ar[d]_{\bar{\Omega}_A} & K_0(B) \ar[d]^{\bar{\Omega}_B} \\ \ar[r]^{\bar{G}} K_i(A) & K_0(B)}$$

\end{obs}

\begin{coro}\label{bar{K}_1}

Let $A$ and $B$ be Artin algebras, and $G : \bar{K}_i(A) \rightarrow \mod B$ an exact additive functor such that $F(P)$ is direct summand of $B^n$ for every $P$ projective $A$-module in $\bar{K}_i(A)$. Assume that $\bar{G}$ is a monomorphism, then:    

\begin{enumerate}

\item  $\phi_A (M) \leq \phi_B(G(M))$ for all $M \in \bar{K}_i(A)$,
\item  $\pd_A M = \pd_B G(M)$ for all $M \in \bar{K}_1(A)$ such that $\pd_A M < \infty$,
\item  $\psi_A (M) \leq \psi_B(G(M))$ for all $M \in \bar{K}_i(A)$.

\end{enumerate}

\end{coro}

\begin{coro}

Let $A$ and $B$ be Artin algebras such that $A$ is not self-injective. If $G : \bar{K}_i(A) \rightarrow \mod B$ is an exact additive functor such that $G(P)$ is a direct summand of $B^n$ for every $P$ projective $A$-module in $\bar{K}_i(A)$, then
\begin{itemize}

\item $\fidim A \leq \fidim B + i$, and

\item $\psidim A \leq \psidim B + i$.

\end{itemize}

\begin{proof}
Let $M$ be an $A$-module, then $\phi_A(M) \leq \phi_A(\Omega_A^{i}(M))+i$. Thus, by Corollary \ref{bar{K}_1}, we obtain $\phi_A (M) \leq \phi_A(\Omega_A^{i}(M)) + i \leq \phi_B(F(\Omega_A^{i} (M))) + i \leq \fidim (B) + i$.
\end{proof}
\end{coro}

\subsection{Aplications}

\begin{prop}

let $A = \frac{\Bbbk Q}{J^2}$ be a radical square zero algebra and $C$ a full subquiver of $Q$ closed by successors, then 
$$\fidim \left(\frac{\Bbbk C}{J_C^2}\right) \leq \fidim (A)\mbox{ and }\psidim \left(\frac{\Bbbk C}{J_C^2}\right) \leq \psidim (A).$$

\begin{proof}

The functor $\iota: \frac{\Bbbk C}{J_C^2} \rightarrow A$ is an exact functor, it sends projective $\frac{\Bbbk C}{J_C^2}$-modules in projective $A$-modules and the induced group morphism $\bar{\iota} :K_0\left(\frac{\Bbbk C}{J_C^2}\right) \rightarrow K_0(A) $ is a monomorphism. Therefore, the thesis follows from Theorem \ref{teorema iny}.

\end{proof}

\end{prop}

\subsubsection{One point extension algebras}

Given two algebras $S$ and $T$ and a $S-T$-bimodule $M$, one constructs the  \textbf{upper triangular matrix algebra} $A$, which is the set of matrices:

$$ \left\{ \left(
      \begin{array}{cc}
        s & m \\
        0 & t \\
      \end{array}
    \right): s\in S, \ m \in M \hbox{ and } t\in T
 \right\}, $$
with the  usual addition and multiplication.

%{\bf Que quiere decir un algebra "in the upper triangular form"?}
\begin{prop}
Given an upper triangular matrix algebra $A$ as above, we have
$$\fidim (T) \leq \fidim (A) \text{ and } \psidim (T) \leq \psidim (A).$$ 
\end{prop}
\begin{proof}
Is a consequence of Theorem \ref{teorema iny} because $\mod T \subset \mod A$.
\end{proof}

In particular, when S is a division algebra we say that $A$ is the one point extension of $T$ by the bimodule $M$. If also $T$ is a bounded path algebra with quiver $Q$ and admissible ideal $I$, then $A$ is a bounded path algebra with quiver $Q'$ and admissible ideal $I'$ such that:

\begin{itemize}

\item $Q'_0 = Q_0 \cup \{v\}$,

\item $Q'_1 = Q_1 \cup \{\alpha_1, \ldots , \alpha_k \}$ where $s(\alpha_i) = v$ and $t(\alpha_i) \in Q_0$ for $1 \leq i \leq k$,

\item $I' = \langle \rho_1, \ldots , \rho_s, \rho_{s+1}, \ldots , \rho_t \rangle$ where $I = \langle \rho_1, \ldots , \rho_s \rangle$ and $s(\rho_m) = v$ for $s+1 \leq m \leq t$.

\end{itemize}

\begin{prop}\label{1.ext}

Given an algebra $A$ such that $A$ is a one point extension algebra of $T$, then $\fidim (T) \leq \fidim (A) \leq \fidim (T) + 1$ and $\psidim (T) \leq \psidim (A) \leq \psidim (T) + 1$. 

\begin{proof}

It follows from Theorem \ref{teorema iny} because $\mod T \subset \mod A$ and from Corollary \ref{bar{K}_1} because $K_1(T) \subset \mod A$.

\end{proof}

\end{prop}

\section{Admisible values and gaps for $\phi$}
\subsection{Admisible values}

Given an algebra $A$ we say that a value $t\in\N$, with $t\leq\fidim(A)$, is admisible if there exists an $A$-module $M$ such that $\phi(M) = t$. If $t$ is not admisible, we say we have a gap at $t$.

We prove that for any algebra $A$ the value $1$ is always admisible. And if $A$ is of finite $\phi$-dimension $m$ then the value $m-1$ is admisible.

\begin{lema}\label{K_1 aditiva}

If $[M] \in K_1(A)$ then for $M' \subseteq M$, $[M'] \in K_1(A)$.
\begin{proof}

Let $M$ be an $A$-module such that $[M] \in K_1(A)$, then there exists $N$ such that $\Omega_A(N) = M$, this means that we have a short exact sequence:

$$ \xymatrix{0 \ar[r]& M \ar[r]& P \ar[r]& N \ar[r]& 0}$$

\noindent where $P$ is a projective module. Let $M'$ be a submodule of $M$, given that the composition $M' \hookrightarrow M \hookrightarrow P$ is a monomorphism, we can consider the next short exact sequence:

$$ \xymatrix{0 \ar[r]& M' \ar[r]& P \ar[r]& \frac{P}{M'} \ar[r]& 0} $$

\noindent following the thesis.

\end{proof}

\end{lema}
\begin{obs}\label{obs_global_finite}
Given an algebra $A$, if there exist $M \in \mod A$ such that $\pd (M) = m \in\N^+$ then $\phi\left(\Omega^{m-i}(M)\right) = i$, for $1\leq i < m$.
In particular, if $A$ has finite global dimension then $\phi(M) = \pd(M)$ for all $M\in\mod A$. Therefore, for all $0\leq i \leq \gd(A)$ there exists $M_i\in\mod(A)$ such that $\phi\left(M_i\right) = i$.
\end{obs}

\begin{teo}\label{teo.1.n-1}

Let $A$ be a finite dimensional algebra. If $\fidim (A) > 0$ then there exist $M\in \mod A$ such that $\phi (M) = 1$. In the case that $\fidim(A)$ is finite, then there exist $N\in  \mod A$ such that $\phi (N) = \fidim(A) - 1$.  

\begin{proof}

Suppose that $A$ is an algebra of infinite global dimension. In the other case, there is a module $M$ with $0 < \pd M = \gd (A) < \infty$ and the result follows from Remark \ref{obs_global_finite}.

Recall the Nakayama rule $\nu$ from Theorem \ref{Nakayama}.

\begin{itemize}

\item First we prove that there exist a module $M$ with $\phi (M) = 1$. Assume that $\pd (M) = \infty$ or $0$ for every $M \in \mod A$. By Theorem \ref{Phi = 0} $A$ is not a self-injective algebra, then applying Theorem \ref{Nakayama} we get that $\nu$ is not a permutation. In this situation there are two possible cases:\\

\begin{enumerate}

\item $\nu$ is a function but is not injective. In this case there exist simple modules $S_1$ and $S_2$ such that $\nu(S_1) = \nu(S_2)$.

Let $M_1$ and $M_2$ be the indecomposable non projective modules (the top of $M_i$  is $S_i$ for $i= 1,2$) given by:

$$\xymatrix{0 \ar[r]& \nu(S_1) \ar[r]& P_1 \ar[r] & M_1 \ar[r]& 0},$$

$$\xymatrix{0 \ar[r]& \nu(S_2) \ar[r]& P_2 \ar[r] & M_2 \ar[r]& 0},$$

therefore $\phi (M_1 \oplus M_2 ) = 1$ because $\pd M_i = \infty$ for $i =1,2$. 

\item Now, if $\nu$ is not a function, then there is a simple module $S$ such that $\soc(P(S))$ is not a simple module.

If $S_1\oplus S_2$ is a direct summand of $\soc(P(S))$ with $S_1$ and $S_2$ non isomporphic simple modules, then there exist two indecomposable projective modules $P_1$ and $P_2$ such that $S \subseteq \soc(P_1) \cap \soc (P_2)$. Therefore we can construct a pair of non-isomorphic modules, $M_1$ and $M_2$, such that $\phi (M_1 \oplus M_2 ) = 1$, similarly to the previous case.

Otherwise $\soc(P(S)) = S'^k$ with $S'$ a simple module and $k \geq 2$. Let $M_1$ and $M_2$ be the following indecomposable modules given by: 

$$\xymatrix{0 \ar[r]& S' \ar[r]& P_1 \ar[r] & M_1 \ar[r]& 0},$$

$$\xymatrix{0 \ar[r]& S'^{2} \ar[r]& P_2 \ar[r] & M_2 \ar[r]& 0}.$$

It is clear that $M_1$ and $M_2$ are indecomposable non-isomorphic modules and $\phi(M_1 \oplus M_2) = 1$.
\end{enumerate}

\item Now suppose $\fidim(A) = m < \infty$. Let $N$ be an $A$-module such that $\phi(N) = m$, then we have to prove that $\phi(\Omega(N)) = m-1$.\\

By Proposition \ref{Huard1} $\phi (\Omega (N)) \geq m-1$. Suppose that $\phi(\Omega(N)) > m-1$. Because $\phi \dim(A) = m$ we have that $\phi(\Omega(N)) = m$. Consider the decomposition  into indecomposable modules $\oplus^{s}_{i} M^{k_i}_i = \Omega(N)$. Using Lemma \ref{K_1 aditiva} we obtain that $[M_i] \in K_1$, and for each $i +1,\ldots , s$ there exists an indecomposable module $N_i \in \mod A$ such that $\Omega(N_i) = M_i$. In this case the $A$-module $N' = \oplus^{s}_{i} N_i$ would have $\phi(N') = m+1$ which is absurd.

\end{itemize}

\end{proof}

\end{teo}

\subsection{Radical square zero algebras with gaps}

It is easy to determine the $\phi$ function in $\mod A$ for $A$ a radical square zero algebra, this allows us to construct examples of algebras with gaps
The next technical proposition helps us with a simplier way to compute gaps. 
\begin{prop}\label{nucleos}
Let $A = \frac{\Bbbk Q}{J^2}$ be a radical square zero algebra such that $\fidim(A)\geq 2$, then $\phi \left(\oplus_{i=1}^t M_i\right) \geq k \geq 2$ if and only if there exist $v \in \ker T^{l-1} \backslash \ker T^{l-2}$ such that $v \in \langle \{\bar{\Omega} ([M_i])\}_{i=1}^t \rangle$ and $l\geq k$. 
\end{prop}
\begin{proof}
If there exist $v \in \mathbb{Q}^n$ such that $v \in (\ker T^{k-1} \backslash \ker T^{k-2}) \cap \langle \{\bar{\Omega} ([M_i])\}_{i=1}^t \rangle$, then
$$\rk \bar{\Omega}^k \langle \{[M_i]\}_{i = 1}^t \rangle = \rk T^{k-1} \langle \{{\bar\Omega}[M_i]\}_{i = 1}^t \rangle < \rk T^{k-2} \langle \{\bar{\Omega}[M_i]\}_{i = 1}^t \rangle = \rk \bar{\Omega}^{k-1} \langle  \{[M_i]\}_{i = 1}^t \rangle,$$  
because 
\begin{itemize}
\item $\sum_{i=1}^t \alpha_iT^{k-1}\left(\bar{\Omega}[M_i]\right) = 0$ and
\item $\sum_{i=1}^t \alpha_iT^{k-2}\left(\bar{\Omega}[M_i]\right) \neq 0$,
\end{itemize}
where $v =  \sum_{i=1}^t \alpha_i(\bar{\Omega}[M_i])$. Therefore $\phi(\oplus_{i=1}^t M_i) \geq k$.

Now suppose that $\phi (M)= l \geq k \geq 2$, then there exist $u = (\alpha_1, \ldots, \alpha_t) \in \mathbb{Q}^n$ such that:

\begin{itemize}

\item $\sum_{i=1}^t \alpha_i(\bar{\Omega}^{l}[M_i]) = \sum_{i=1}^t \alpha_iT^{l-1}(\bar{\Omega}[M_i]) = 0$ and

\item $\sum_{i=1}^t \alpha_i(\bar{\Omega}^{l-1}[M_i]) = \sum_{i=1}^t \alpha_iT^{l-2}(\bar{\Omega}[M_i]) \neq 0$,

\end{itemize}
then $v =  \sum_{i=1}^t \alpha_i(\bar{\Omega}[M_i]) \in \ker T^{l-1} \backslash \ker T^{l-2}$.
\end{proof}

As a consequence of the previous result we obtain

\begin{coro}\label{coro_nucleos}
Let $A = \frac{\Bbbk Q}{J^2}$ be a radical square zero algebra with $\fidim(A)\geq 2$. If $\phi (\oplus_{i=1}^t M_i) \geq 2$, then 
$$\phi (\oplus_{i=1}^t M_i) = \sup\{k: \hbox{ there exist } v \in (\ker T^{k-1} \backslash \ker T^{k-2}) \cap \langle \{\bar{\Omega} ([M_i])\}_{i=1}^t \rangle \}$$ 
\end{coro}

Now, we can show an algebra with a gap.
%The following example shows that if a finite dimensional algebra $A$ has $\fidim (A) = m > 3$, then might exists $1 < h < m-1$ such that there is no $A$-module $M$ with $\phi(M) = h$. In this case we say that $A$ has a {\bf gap}.

\begin{ej}\label{ejemplo_gap}

Let $A = \frac{\Bbbk Q}{J^2}$, where $Q$ is the following quiver:

\vspace{.5cm}

%$$\xymatrix{ & \cdot^0 \ar[dr] \ar[dl] & \\ \cdot^1 \ar[d] & & \cdot^2 \ar[d] \\ \cdot^3 \ar[d] & & \cdot^4 \ar[d] \\ \cdot^5 \ar[d] & & \cdot^6 \ar[d] \\ \cdot^7 \ar[rruuu] \ar[d] & & \cdot^8 \ar[lluuu]  \ar[d] \\ \cdot^9 \ar[d] & & \cdot^{10} \ar[d] \\ \cdot^{11} \ar@/^2pc/[uuu] \ar[d] & & \cdot^{12} \ar@/_2pc/[uuu] \ar[d] \\ \cdot^{13}\ar[dr] & & \cdot^{14} \ar[dl] \\ & \cdot^0 & \\}$$

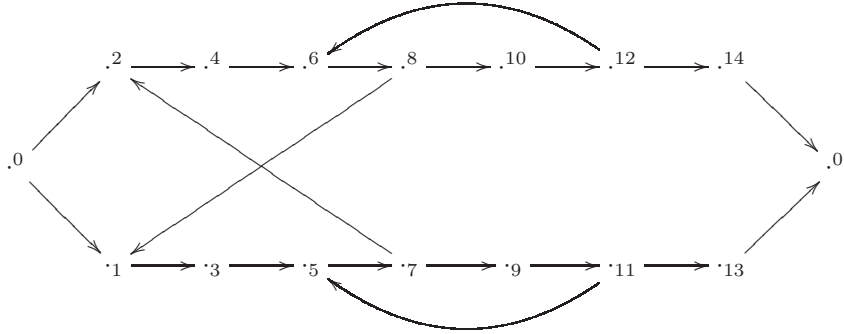
$$\xymatrix{ &\cdot^2 \ar[r] & \cdot^4 \ar[r] & \cdot^6 \ar[r]& \cdot^8 \ar[ddlll] \ar[r]& \cdot^{10} \ar[r]& \cdot^{12} \ar[r] \ar@/_2pc/[lll]& \cdot^{14} \ar[dr]& \\ 
\cdot^0 \ar[dr] \ar[ur] & & & & & & & & \cdot^0\\ & \cdot_1 \ar[r]& \cdot_3 \ar[r]& \cdot_5 \ar[r]& \cdot_7 \ar[r] \ar[uulll]& \cdot_9 \ar[r]& \cdot_{11} \ar[r] \ar@/^2pc/[lll]& \cdot_{13} \ar[ur]& }$$
\vspace{.5cm}

Using Theorem 4.32 from \cite{LMM} and the following facts, we conclude that $\fidim(A) = 8$:

$$\begin{array}{ll}

 \ker T = \langle w_7 \rangle, & T(v_0) = v_1, \\

 \ker T^2 = \langle w_6-w_2 \rangle + \ker T, & T(v_1) = v_2,\\

 \ker T^3 = \langle w_5-w_1 \rangle + \ker T^2,& T(v_2) = v_3,\\

 \ker T^4 = \langle w_4 \rangle + \ker T^3, & T(v_3) = v_4,\\

 \ker T^5 = \langle w_3 \rangle + \ker T^4, & T(v_4) = v_1 + v_5,\\

 \ker T^6 = \langle w_2 \rangle + \ker T^5, & T(v_5) = v_6,\\

 \ker T^7 = \langle w_1 \rangle + \ker T^6, & T(v_6) = v_3 + v_7,\\

 											& T(v_7) = 2v_0, \end{array}$$
where $w_i = e_{2i-1}-e_{2i}$ for $i = 1, \ldots, 7$ and $v_i = e_{2i-1}+e_{2i}$ for $i = 1, \ldots, 7$.%, but there is no $A$-module $M$ such that $\phi(M) = 3$. 

Suppose that $M$ is an $A$-module such that $\phi (M) = 3$, then there exist $v \in (\ker T^2 \backslash \ker T) \cap \langle add M \rangle$. This means that $M = M_1 \oplus M_2$ and $v = \alpha \bar{\Omega} ([M_1]) - \beta \bar{\Omega}([M_2])$ with $\alpha, \beta \in \mathbb{Q}^{+}$. Because $v \in \ker T^2 \backslash \ker T$ then $\bar{\Omega}([M_1]) = e_{11}+e_4 + [S] $ and $\bar{\Omega}([M_2])= e_{12}+e_3 + [S']$.

On the other hand, $\Omega (N) = S_i \oplus S$ implies that $S_{i-2}$ is a direct summand of $N$ for $i = 3,4,11,12$. Therefore $S_1 \oplus S_2 \oplus S_9 \oplus S_{10}$ is a direct summand of $M$ and this is absurd, because $e_1-e_2 \in (\ker T^7 \backslash \ker T^6) \cap \langle \add M \rangle$. Concluding that the value $3$ is not admisible.
\end{ej}

In previous articles it has been proven that for some families of algebras $\fidim A = \fidim A^{op}$ (see \cite{LMM}, \cite{BMR} and \cite{LM}). The previous fact rises the question: if $A$ has a gap, does the opposite algebra have it also?
The example below shows us that the answer is no.

\begin{ej}
Let $A = \frac{\Bbbk Q}{J^2}$ where $Q$ is the following quiver,
$$  
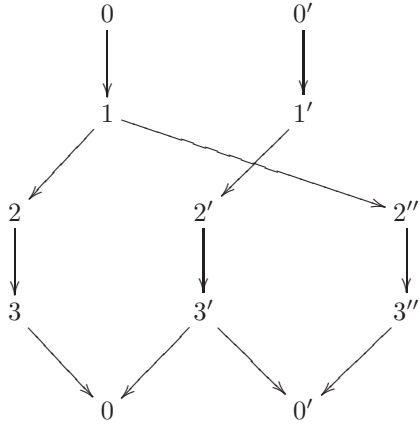
\xymatrix{ & 0 \ar[d] & & 0' \ar[d] & \\
			 & 1 \ar[dl] \ar[drrr] & & 1' \ar[dl] & \\ 
 			2 \ar[d] & & 2'\ar[d] & & 2''\ar[d] \\ 
			3 \ar[dr]& & 3' \ar[dr] \ar[dl] & & 3''\ar[dl] \\ 
			 & 0 & & 0' &  }$$
By Proposition \ref{tata} we have $\fidim(A) = \phi(\oplus_{S \in \mathcal{S}} S_{i}) + 1 = 8$. By easy computations we have that
$$\begin{array}{ccc}
\phi(S_3 \oplus \frac{P_3'}{S_0}) = 1, & \phi(S_1\oplus S_{1'}) = 3, & \phi(S_0\oplus S_{0'}) = 4,\\
\phi(S_3\oplus S_{3''}) = 5, & \phi(S_3\oplus S_{3''}) = 6, & \phi(\oplus_{S \in \mathcal{S}} S_{i}) = 7,\\
\end{array}$$
but there is no $A$-module $M$ such that $\phi(M) = 2$.
On the other hand the opposite algebra $A^{op} = \frac{\Bbbk Q^{op}}{J^2}$ has a associated quiver $Q^{op}$ as follows,
$$  
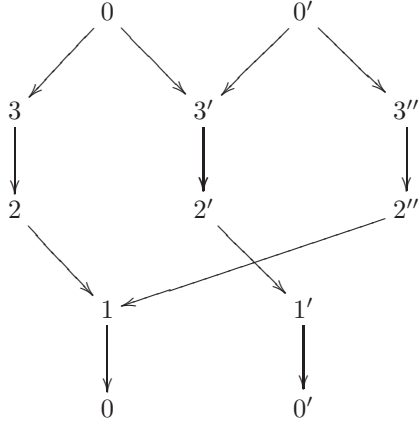
\xymatrix{ & 0 \ar[dr] \ar[dl] & & 0' \ar[dl]\ar[dr] & \\
			3 \ar[d]& & 3' \ar[d] \ar[d] & & 3''\ar[d] \\ 
 			2 \ar[dr] & & 2'\ar[dr] & & 2''\ar[dlll] \\ 
			& 1 \ar[d] & & 1' \ar[d] & \\ 
			 & 0 & & 0' &  }$$
Again, by Proposition \ref{tata} $\fidim(A^{op}) = 8$. It is also clear that:
$$\begin{array}{ccc}
\phi(S_2 \oplus S_{2''}) = 1, & \phi(S_3\oplus S_{3''}) = 2, & \phi(S_0\oplus S_{0'}) = 3,\\
\phi(S_1\oplus S_{1'}) = 4, & \phi(S_2\oplus S_{2'}) = 5, & \phi(S_3\oplus S_{3'}) = 6, \\ & \phi(\oplus_{S \in \mathcal{S}} S_{i}) = 7.&\\
\end{array}$$
Hence $A^{op}$ has no gaps.
\end{ej}

\subsection{Algebras without gaps}
It is well known that $\pd \Omega(M) = \pd M - 1$ for every finite projective dimensional non projective module $M$, as a consequence we have the remark.
\begin{obs}\label{sin gap}
If $0< \pd (M) < \infty $ then there are no gaps between $0$ and $\pd (M)$.
\end{obs}

As a consequence of Remark \ref{sin gap} we get the result.

\begin{teo}

If $A$ is a finite dimensional algebra with gaps, then:
$$\fdim(A)<\lambda< \fidim(A),$$
for every gap $\lambda$. Hence the finitistic conjecture holds for $A$.
\end{teo}

\begin{ej}
By the previous remark we have that the following families of algebras have no gaps.
\begin{itemize}
\item Finite global dimensional algebras.
\item Gorenstein algebras (See Theorem 4.7 of \cite{LM}). 
\end{itemize}
\end{ej}
From Theorem \ref{teo.1.n-1} it follows that every algebra with $\phi$-dimension less or equal to $3$ has no gaps. In particular if $A = \frac{\Bbbk Q}{J^k}$ is a truncated path algebra we have the following example.
\begin{ej}\label{ej_menor_3}\ 
\begin{itemize}
\item In case $k$ is large enough ($k \geq n-1 \geq 2$), the truncated algebra $A$ is an example of algebra with $\phi$-dimension less or equal to $3$ (See Remark 5.4 of \cite{BMR}).
\item If $Q$ is a symmetric quiver then its $\phi$-dimension is lower or equal to $2$. (See Corollary 4.62 of \cite{LMM} and Theorem 4.17 of \cite{BMR})
\end{itemize}
\end{ej}
The following result allows us to build new examples of algebras without gaps.
\begin{prop}
If $A$ is a finite dimensional algebra without gaps and $T$ is its one point extension, then $T$ has no gaps. 
\end{prop}
\begin{proof}
It is a consequence of Corollary \ref{iny+ind}. 
\end{proof}

We now give conditions on the quiver $Q$ for the radical square zero algebra $\frac{\Bbbk Q}{J^2}$ to have no gaps.
\begin{obs}\label{D_4tilde}
We consider the euclidian quiver $\tilde{D}_{4}$ with the following orientation:
$$
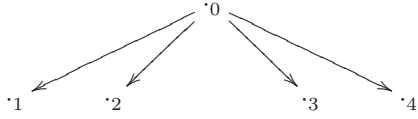
\xymatrix{ & & \cdot_0 \ar[dll] \ar[dl] \ar[dr] \ar[drr] & & \\ \cdot_1 & \cdot_2 & & \cdot_3 & \cdot_4 }$$
There is family of indecomposable modules $\{M_n\}_{n \in \N}$ with this shape:
$$
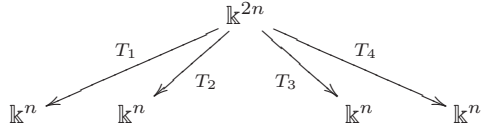
\xymatrix{ & & \Bbbk^{2n} \ar[dll]_{T_1} \ar[dl]^{T_2} \ar[dr]_{T_3} \ar[drr]^{T_4} & & \\ \Bbbk^n & \Bbbk^n & & \Bbbk^n & \Bbbk^n }$$
See proof Panoromic view (3) of \cite{R} and Lemma 3.3.2.
\end{obs}

\begin{obs}
If $Q$ has a vertex $v_0$ with outdegree bigger or equal to $4$ then we have an embedding of categories $\iota: \mod \Bbbk \tilde{D}_4 \rightarrow \mod \frac{\Bbbk Q}{J^2}$.
\end{obs}

\begin{nota}
Given a radical square zero algebra $A$, we denote by $\Gamma_A$ the separated quiver of $A$ (see chapter X.2 \cite{ARS}).
\end{nota}

\begin{obs}
Let $Q$ be a quiver and $A = \frac{\Bbbk Q}{J^2}$ its associated radical square zero algebra. Then there exist a functor (See Lemma 2.1 chapter X.2 \cite{ARS})
$$F : \mod A \rightarrow \mod \Bbbk \Gamma_A,$$
such that the following properties holds:
\begin{itemize}
\item $M$ is an indecomposable $\frac{\Bbbk Q}{J^2}$-module if and only if $F(M)$ is an indecomposable $\Bbbk \Gamma_A$-module,
\item $M$ and $N$ are isomorphic if and only if $F(M)$ and $F(N)$ are isomorphic,
\item $\dim_{\Bbbk} M = \dim_{\Bbbk} F(M)$.
\end{itemize}
If in addition there is a vertex $v_0$ with outdegree bigger or equal to $4$ then there exists a family of $A$-modules $\{N_n\}_{n\in \N}$ such that $F(N_n) = M_n$.
\end{obs}

\begin{lema}\label{todadim}
Let $Q$ be a quiver where every vertex has outdegree bigger or equal to $6$ and for two different vertices of $Q_0$, $v$ and $w$, there exist $4$ different arrows $\alpha_1, \alpha_2, \beta_1, \beta_2$ such that: 
\begin{itemize}

\item $s(\alpha_1) = s(\alpha_2) = v$, 

\item $s(\beta_1) = s(\beta_2) = w$,

\item $t(\alpha_1) = t(\beta_1)$,

\item $t(\alpha_2) = t(\beta_2)$.

\end{itemize}
Consider $A = \frac{\Bbbk Q}{J^2}$, then for every $v = (v_1,\ldots, v_n) \in \mathbb{N}^n$ there exists an indecomposable $A$-module $M_v$ and $l \in \mathbb{Z}^+$ such that $\bar{\Omega}(l[M_v]) = \oplus \alpha_i [S_i]$ and $T(v) = \oplus_i^n \alpha_i e_i$. In particular if $v_i > 0$ for at least two indices,  then $l = 1$.

\begin{proof}
If $v$ has only one entry $v_{i_0}$ non-null, choose $M_v = S_{i_0}$ and $l=v_{i_0}$. If there are more than one entry no null, suppose $v = \sum n_ie_i$ where $v_1, \ldots, v_s$ are the ordered vertices with $s \geq 2$ and non-null cofficients $\{n_i\}$. Consider the following sets of arrows
$$A =  \{\alpha_1, \ldots, \alpha_{s-1}, \beta_2, \ldots, \beta_s\},$$
such that:
\begin{itemize}

\item $s(\alpha_1) = v_1$ and $s(\beta_s) = v_s$,

\item $s(\alpha_i) = s(\beta_i) = v_i$ and $\alpha_i \neq \beta_i$ for $i= 2, \ldots, s-1$,

\item $t(\alpha_i) = t(\beta_{i+1}) = w_i$ for $i = 1, \ldots, s-1$.

\end{itemize}

For $i= 1, \ldots, s$, we define $A_i = \{\alpha_{i,j}\}_{j = 1, 2, 3, 4}$, such that
\begin{itemize}  
\item $\alpha_{i,j} \notin A$ for $i= 1, \ldots, s$ and $j = 1, 2, 3, 4$, 
\item $s(\alpha_{i,j}) = v_i$ for $i= 1, \ldots, s$ and $j = 1, 2, 3, 4$, 
\item $\alpha_{i,j} \neq \alpha_{i,j'}$ for $i= 1, \ldots, s$ and $j \neq j'$.
\end{itemize}

Let $B$ be the set
$$B = \{ \gamma \in Q_1 \text{ : } s(\gamma) \in \{v_1,\ldots, v_s\} \text{
and } \gamma \notin (A \cup A_1 \cup \ldots \cup A_s)\}.$$

Let $G = (G_0, G_1)$ be the following quiver:

\begin{itemize}
\item $G_0 = \{ v_i \text{ : } i = 1,\ldots, s \} \cup \{ v_{\gamma} \text{ : } \gamma \in B  \cup A_1\cup \ldots \cup A_s \} \cup \{ w_{i} \text{ : } i = 1, \ldots, s-1\},$
\item $G_1$ is the union of the sets
\begin{itemize}
\item[*] $\{\delta_{\gamma} \text{ : } s(\delta_{\gamma}) = s(\gamma) \text{ and } t(\delta_{\gamma}) = v_{\gamma}  \text{ where } \gamma \in B  \cup A_1\cup \ldots \cup A_s\},$
\item[*] $\{ \bar{ \alpha_i} \text{ : }s(\bar{ \alpha_i}) = v_i \text{ and } t(\bar{ \alpha_i}) = w_{i} \text{ for } i =1,\ldots, s-1\},$
\item[*] $\{ \bar{ \beta_j} \text{ : }s(\bar{ \beta_j}) = v_{j} \text{ and } t(\bar{ \beta_j}) = w_{j-1}\text{ for } j =2, \ldots, s\}.$
\end{itemize}
\end{itemize}
Then there exist an embedding functor $H:\mod \Bbbk G \rightarrow \mod \Bbbk \Gamma_Q$ such that 
\begin{itemize}
\item $H(C)$ is indecomposable in $\mod \Bbbk \Gamma_Q$ if and only if $C$ is indecomposable in $\mod \Bbbk G$,
\item $\dim_{\Bbbk} H(C) = \dim_{\Bbbk} C$.
\end{itemize}

Now, consider $M = (M_v,T_{\alpha})_{v \in G_0, \alpha \in G_1}$ the following indecomposable $G$-module:
\begin{itemize}
\item $M_v = \left\{\begin{array}{ll}
\Bbbk^{2n_i} & \text{ if } v=v_i \text{ for } i =1,\ldots, s, \\
\Bbbk^{n_i}  & \text{ if } v = v_{\gamma} \text{ for } \gamma \in  B  \cup A_1\cup \ldots \cup A_s,\\
\Bbbk^{n_i+ n_{i+1}} & \text{ if } v = w_i \text{ for } i = 1,\ldots, s-1,\\
\end{array} \right.$ 
\\
\item $T_{\alpha} = \left\{\begin{array}{ll} 
T_i\ (\text{see } \ref{D_4tilde} )& \text{ if } \alpha = \delta_{\gamma} \text{ with } \gamma = \alpha_{i,j} \in A_1\cup \ldots \cup A_s,\\
\text{an epimorphism } & \text{ if } \alpha = \delta_{\gamma} \text{ with } \gamma \in B,\\
\hspace{-5pt}
\begin{array}{l}
\text{an epimorphism}\\
\text{(monomorphism)}
\end{array}& \text{ if } \alpha = \bar{ \alpha_i} \text{ and } n_i \leq n_{i+1}(n_{i+1} \leq n_{i}),\\
\hspace{-5pt}
\begin{array}{l}
\text{an epimorphism}\\
\text{(monomorphism)}
\end{array} & \text{ if } \alpha = \bar{ \beta_j} \text{ and } n_{j} \leq n_{j-1}(n_{j-1} \leq n_{j}).
\end{array}\right.$
\end{itemize}
Then $HF(M)$ verifies the thesis.
\end{proof}
\end{lema}

With similar computations as the previous theorem, we can prove.

\begin{lema}\label{flechadoble}
Let $Q$ be a quiver where every vertex has outdegree bigger or equal to $4$ with a double arrow and for two different vertices of $Q_0$, $v$ and $w$, there exist $4$ different arrows $\alpha_1, \alpha_2, \beta_1, \beta_2$ such that: 
\begin{itemize}
\item $s(\alpha_1) = s(\alpha_2) = v$, 
\item $s(\beta_1) = s(\beta_2) = w$,
\item $t(\alpha_1) = t(\beta_1)$,
\item $t(\alpha_2) = t(\beta_2)$.
\end{itemize}
Consider $A = \frac{\Bbbk Q}{J^2}$, then for every $v = (v_1,\ldots, v_n) \in \mathbb{N}^n$ with $v_i\geq 0$ there exists an $A$-module $M_v$ and $l\in\Z^+$ such that $\bar{\Omega}(l[M_v]) = \oplus \alpha_i [S_i]$ and $T(v) = (\alpha_1, \ldots, \alpha_n) $. In particular if $v_i > 0$ for at least two indices,  then $l = 1$.
\end{lema}

As a consequence of the two previous lemmas and the Corollary \ref{coro_nucleos} we have the following theorem.
\begin{teo}
If $Q$ is a quiver in the hypothesis of Lemma \ref{todadim} or Lemma \ref{flechadoble} then $A = \frac{\Bbbk Q}{J^2}$ has no gaps.
\end{teo}

The previous theorem allows us to construct examples not considered in Example \ref{ej_menor_3}(See Remark 3.2 of \cite{BMR}).
\begin{obs}
Given a quiver $Q$ with $n$ vertices, there exist a qiver $Q'$ such that $Q_0 = Q'_0$, $Q$ is a subquiver of $Q'$ and $\fidim( \frac{\Bbbk Q'}{J^2}) = n$. In particular we can construct radical square zero algebras with  maximal $\phi$-dimension and where their asociated quiver verifies the hypothesis of Lemma \ref{todadim}.
\end{obs}

\section{Partial $\phi$-dimensions}

We define the partial $\phi$-dimension for a natural $l$ and restate, in Proposition \ref{prop.autoiny.marzelo}, the main result of \cite{HL} in our context. We also give examples in wich we compute the partial $\phi$-dimensions.

\begin{defi}
Given $l\in\Z^+$ we define the $l$ $\phi$-dimension of an algebra $A$ as:
\[\fidim_l (A) = \sup \left\{ \phi(M_1\oplus \ldots \oplus M_l): M_i \in \ind A \ \forall i = 1, \ldots, l \right\}.\]

\end{defi}

\begin{obs}\label{desigualdades}

Given an algebra $A$ we have:
\begin{enumerate}
\item $\fdim (A) = \fidim_1 (A)$,
\item for all $l\in\Z^+$ $\fidim_l(A)\leq \fidim_{l+1}(A) \leq \fidim(A)$,
\item if $\fidim(A)$ is finite, then there exists $l \in \Z^+$ such that $\fidim_l (A) = \fidim (A)$.
%$$\fdim (A) = \fidim_1 (A) \leq \fidim_2 (A) \leq \ldots \leq \fidim_l (A) = \fidim (A).$$
\end{enumerate}
\end{obs}

\begin{prop}\label{prop.autoiny.marzelo}
Let $A$ be a finite dimensional algebra. The following statements are equivalent:
\begin{enumerate}
\item $A$ is a self-injective algebra.
\item $\fidim (A) = 0$.
\item There exists $l \in \Z^+$ such that $\fidim_l(A) = 0$.
\item $\fidim_2(A) = 0$.
\end{enumerate}
\end{prop}

\begin{proof}
By Corollary $6$ of \cite{HL} the first two statements are equivalent. %That the second statement implies the third is Remark \ref{desigualdades} part 
By Remark \ref{desigualdades} part (2) we have that the second statement implies the third one, and also that the third one implies the fourth one.

To see that the fourth statement implies the second one, by the proof of Theorem \ref{teo.1.n-1} if $\fidim (A) > 0$ then there exist two indecomposable modules $M_1$, $M_2$ such that $\phi(M_1\oplus M_2) = 1$
\end{proof}

In some families of algebras all partial dimensions coincide. For a proof of the following result see Theorem 4.7 of \cite{LM}.

\begin{ej}

If $A$ is an $m$-Gorenstein algebra, then: 

$$\fidim_1(A)= \fidim_2(A) = \cdots = \fidim(A)\leq m.$$

\end{ej}

\subsection{Partial $\phi$ dimension for radical square zero algebras}

The following two results about partial dimensions can be found in \cite{LMM}.

\begin{prop}\label{prop.LMM.4.14}

If $A = \frac{\Bbbk Q}{J^2}$ is a radical square zero algebra with $\vert Q_0 \vert = n$, then $\fidim (A) = \fidim_{n}(A)$. 

\begin{proof}

See Proposition 4.14 of \cite{LMM}.\end{proof}

\end{prop}

\begin{prop}\label{fimaxima}

If $A = \frac{\Bbbk Q}{J^2}$ is a radical square zero algebra with $\vert Q_0 \vert = n$, and $\fidim (A) = n$ then $\fidim_2 (A) = n$.

\end{prop}

\begin{proof}

See Corollary 5.7 of \cite{LMM}.\end{proof}

Now, we give a generalization of Proposition \ref{fimaxima}.

\begin{prop}

Let $A = \frac{\Bbbk Q}{J^2}$ be a radical square zero algebra with $\vert Q_0 \vert = n$ and without sinks nor sources. If $\fidim (A) = k \leq n$ then $\fidim_{n-k+2}(A) = k$.

\begin{proof}
We can assume that $k\geq3$ because the other cases were considered in Proposition \ref{prop.LMM.4.14}.
Since $\fidim (A) = k$ and $\fidim (A) = \phi \left(\oplus_{S\in \mathcal{S}}S\right) + 1$, then
$$\rk T^{k-2} > \rk T^{k-1} = m_0 \leq n-(k-1).$$ 
This implies that there exist indices $i_1,i_2\dots, i_{m_0+1}$ such that the set\\
$A = \left\{T^{k-2}\left(e_{i_1}\right), T^{k-2}\left(e_{i_2}\right), \ldots ,T^{k-2}\left(e_{i_{m_0+1}}\right) \right\}$ 
is linearly independent, and the set $T(A)$ is linearly dependent.
Thus $\phi\left(S_{i_1}\oplus S_{i_2} \oplus \ldots \oplus S_{i_{m_0}+1}\right) \geq k-1$.
Finally, if we consider $M_j \in \mod A$ such that $\Omega\left(M_j\right) = S_j$, then $\phi\left(\oplus_{j = 1}^{m_0 +1} M_j\right) = k$.
\end{proof} 
\end{prop}

\begin{ej}\label{ejemplo_todas_las_dimensiones}
Let $\Gamma^m$ be the following quiver with $|\left(\Gamma^m\right)_0| = n$ and without sinks nor sources,

$$
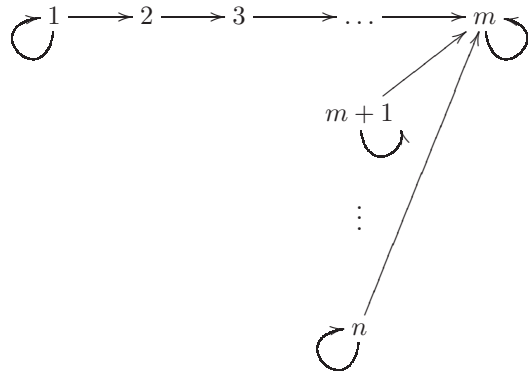
\xymatrix{  1 \ar@(d,l) \ar[r]& 2 \ar[r] & 3 \ar[r]& \ldots \ar[r]& m \ar@(d,r)\\
&&& m+1 \ar@(d,r) \ar[ur]&\\
&&& \vdots &\\
&&&n \ar@(d,l) \ar[uuur]}$$ 

\vspace{.5cm}

then $\fidim_2 \left(\frac{\Bbbk \Gamma^m}{J^2}\right) = \fidim \left(\frac{\Bbbk \Gamma^m}{J^2}\right)  = m-1$ for $ 0 < m \leq n$.

\end{ej}

As a consequence of the Lemma \ref{todadim} (or Lemma \ref{flechadoble}) and Proposition \ref{nucleos} we obtain the following result.

\begin{coro}
In the same hypothesis of Lemma \ref{todadim} or Lemma \ref{flechadoble}, for $1 \leq i \leq \fidim(A)-1$ there exist indecomposable modules $M^i_1$ and $M^i_2$ such that $\phi(M^i_1 \oplus M^i_2) = i$. In particular $\fidim_2(A) \geq \fidim (A)-1$.
\end{coro}
\begin{proof}
First we notice that every non-projective $A$-module has infinite projective dimension. 
Next we consider $v \in \mathbb{Z}^n \cap  (\ker T^{i-1} \backslash \ker T^{i-2})$ for $ 1 \leq i \leq n-1 $, then there exist $v^+$, $v^- \in \mathbb{N}^n$ such that $v = v^+-v^-$.  By Lemma \ref{todadim}, there are indecomposable  modules $M_+$, $M_{-}$ and $l^+, l^- \in \mathbb{Z}^+$ such that $\Omega\left({M_+}^{l^+}\right) = \sum_{i = 1}^n \alpha_i[S_i]$, $T(v^+) = \sum_{i = 1}^n \alpha_i e_i$ and $\Omega\left({M_-}^{l^-}\right) = \sum_{i = 1}^n \beta_i [S_i]$, $T(v^-) = \sum_{i = 1}^n \beta_i e_i$. Hence $\phi(M_+\oplus M_{-}) = i$. \end{proof}

The following example shows that the inequalities in Remark \ref{desigualdades} could be strict.

\begin{ej}

Consider the radical square zero algebra $A = \frac{\Bbbk Q}{J^2}$, whose quiver $Q$ is the following:
$$
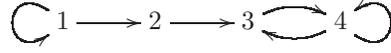
\xymatrix{ & 1 \ar@(ul,dl) \ar[r]& 2 \ar[r] & 3 \ar@/^/[r]& 4 \ar@(dr,ur) \ar@/^/[l]}$$

\vspace{.5cm}
Then $K_1(A)$ has basis $\left\{[S_1], [S_2], [S_3], [S_4]\right\}$, where $S_i$ is the simple module associated to the vertex $i$.
The values of   $\Omega$ in the simple modules are the following:

\begin{itemize}
  \item $\Omega(S_1) = S_1 \oplus S_2$.
  \item $\Omega (S_2) = S_3$.
  \item $\Omega (S_3) = S_4$.
  \item $\Omega (S_4) = S_3 \oplus S_4$.
\end{itemize}

From this we see that  $rk(K_1(A)) = 4 > 3 = rk (K_2(A)) = rk(K_3(A))$, which implies that  $\fidim ( A) = 2$. If we assume that  $M = M_1 \oplus M_2$ is the decomposition of a module $M$ such that $\phi(M) = \fidim A$, we see that  $\left\{ [\Omega(M_1)], [\Omega(M_2)] \right\}$ must be linearly independent and  $\left\{[ {\Omega}^{2}(M_1)], [\Omega^{2}(M_2)] \right\}$ must be colinear. Since  $[S_2 ]+ [S_3] - [S_4]$ is the  basis of the kernel of  $\overline{\Omega}$ we get, without loss of generality, that:

\begin{itemize}
  \item  $S_2 \oplus S_3$ is a direct summand of $\Omega(M_1)$,
  \item  $S_4$ is direct summand of $\Omega(M_2)$.
\end{itemize}

Since $P_1$ is the only projective indecomposable module, up to isomorphism, which has $S_2$ as a submodule and, on the other hand,
$S_3$ is a submodule of the projective modules $P_2$ and $ P_4$, analysing the quiver $Q$, we can infer that $M_1$ cannot be indecomposable.

On the other hand, it is easy to see that $\phi(S_2\oplus S_3 \oplus S_4) = 1$. Because $S_2, S_3, S_4 $ belongs to $K_1$, there exist indecomposable modules $M_1$, $M_2$, $M_3$, such that $\Omega(M_i) = S_i$, thus $\phi(M_1\oplus M_2 \oplus M_3) = 2$.

Finally we conclude that $1 = \fidim_2 \left(\frac{\Bbbk Q}{J^{2}}\right) < \fidim_3 \left(\frac{\Bbbk Q}{J^{2}}\right) = 2$.

\end{ej}

The example below shows us that for all $l \geq 2$ there exist radical square zero algebras $A$ such that: 

$$\fidim_{l}(A) < \fidim_{l+1}(A) = \fidim (A). $$ 

\begin{ej} Let $Q$ be the following quiver:

$$
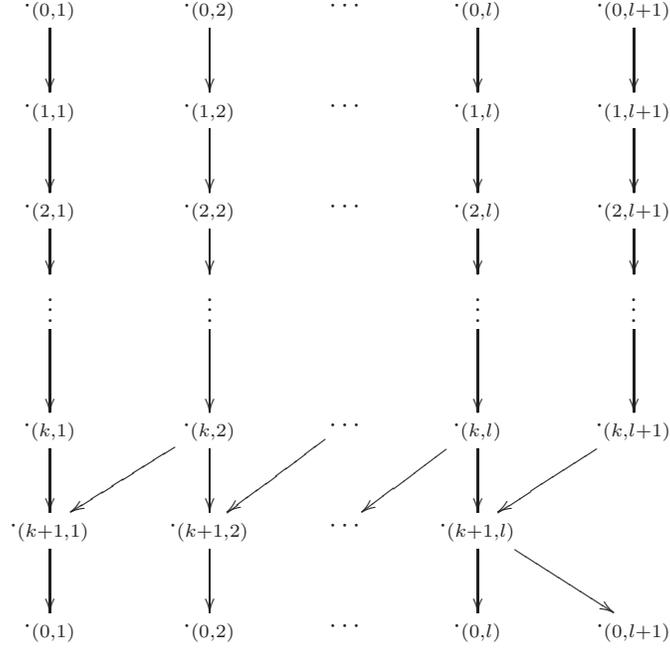
\xymatrix{ \cdot_{(0,1)} \ar[d] &  \cdot_{(0,2)} \ar[d] & \cdots & \cdot_{(0,l)} \ar[d] & \cdot_{(0,l+1)} \ar[d] \\
\cdot_{(1,1)} \ar[d] & \cdot_{(1,2)} \ar[d] & \cdots & \cdot_{(1,l)} \ar[d] & \cdot_{(1,l+1)} \ar[d] \\
\cdot_{(2,1)} \ar[d] & \cdot_{(2,2)} \ar[d] & \cdots & \cdot_{(2,l)} \ar[d] & \cdot_{(2,l+1)} \ar[d] \\
\vdots \ar[d] & \vdots \ar[d] &  & \vdots \ar[d] & \vdots \ar[d] \\ 
\cdot_{(k,1)} \ar[d] & \cdot_{(k,2)} \ar[d] \ar[dl] & \cdots  \ar[dl] & \cdot_{(k,l)} \ar[d] \ar[dl] & \cdot_{(k,l+1)} \ar[dl] \\
\cdot_{(k+1,1)} \ar[d] & \cdot_{(k+1,2)} \ar[d] & \cdots & \cdot_{(k+1,l)}\ar[dr] \ar[d] &  \\
\cdot_{(0,1)}  & \cdot_{(0,2)}  & \cdots & \cdot_{(0,l)}  & \cdot_{(0,l+1)} }$$

If $A = \frac{\Bbbk Q}{J^2}$, then $\fidim_s(A) < \fidim_{l+1}(A) = \fidim (A)$ for $s = 1, \ldots , l$.

Let $V_{i}$ be the the following subspaces of $\mathbb{Q}^n$:

\begin{itemize}

\item $V_i$ has basis $B_{i} = \{ e_{(i,j)} : j \in \{0, \ldots ,l+1 \} \}$ for all $i \leq k$ and

\item $V_{k+1}$ has basis $B_{k+1} = \{ e_{(k+1,j)} : j \in \{0, \ldots ,l\} \}$.

\end{itemize}

Then $T(V_{i}) = V_{i + 1}$ for all $i \leq k$ and $T(V_{k+1}) = V_{0}$. Consider the linear transformation $S= T^{k+1}_{\vert V_{k+1}}: V_{k+1} \to V_{k+1}$, then its associated matrix to the basis $B_{k+1}$ is
 
$$\begin{pmatrix}
1 & 1 & 0 & \dots & 0 & 0 \\ 
0 & 1 & 1 &  & \vdots & \vdots \\ 
\vdots & 0 & 1 &  & 0 & \vdots \\ 
\vdots & \vdots & 0 & \ddots & 1 & 0 \\ 
\vdots & \vdots & \vdots &  & 1 & 1 \\ 
0 & 0 & 0 & \dots & 0 & 2
\end{pmatrix} $$

Hence $S^{k+2}$ is invertible and we deduce that $\phi\left(\oplus_{j = 1}^{l} S_{(k+1,j)}\right) = 0$.

Now $T^{k +1 - i}(V_{i}) \subset V_{k+1}$, and by Proposition \ref{Huard1} we have 
$\phi\left(\oplus_{j = 1}^{l+1} S_{(i,j)}\right) \leq k+1 - i$ for all $i < k+1$. In fact the previous inequality is an equality, it follows from the existence of the vector
$$ w_i = \sum_{j = 1}^{\lfloor \frac{l}{2}\rfloor} e_{(i,2j-1)} -  \sum_{j = 1}^{\lfloor \frac{l-1}{2}\rfloor} e_{(i,2j)} \in Ker(T^{k+1 - i}), $$ 
and because $\ker(T^{k+1 - i}) = \langle w_i \rangle \oplus \ker(T^{k - i})$.

On the other hand, the quiver $Q$ has no sources nor sinks, so we conclude that $\phi dim (A) = k+2$. 

Let $M = \oplus_{j \in J} M_{j}$  be an $A$-module where each $M_{j}$ is indecomposable for $j \in J$  and $\phi(M) = k + 2$. By Proposition \ref{nucleos} we have a vector $\alpha w_0 + v \in \langle \bar{\Omega}[M_j]_{j \in J} \rangle$ with $\alpha \neq 0$ and $v \in \ker T^{k}$. 

For $1 \leq s \leq l-1$ there is an unique idecomposable $A$-module $N_s$ such that $S_{(0,s)} \subset \Omega(N_s)$. On the other hand, $\Omega(S_{(k+1,l)}) = S_{(0,l)}\oplus S_{(0,l+1)}$, and there exist indecomposable modules $N_l$, $N_{l+1}$ such that $\Omega(N_l) = S_{(0,l)}$, $\Omega(N_{l+1}) = S_{(0,l+1)}$, and for any other indecomposable module $N$, $\Omega(N)\cap (S_{(0,l)}\oplus S_{(0,l+1)}) = \{0\}$. Finally, by the last facts $M$ must have at least $l+1$ non-isomorphic direct summands.
\end{ej}

\bibliographystyle{elsarticle-harv}

\end{document}